\documentclass[a4]{amsart}

\usepackage{amsbsy, amsmath, amsfonts, braket, latexsym, amsthm, amsxtra, amssymb, amscd, graphics, mathrsfs, mathtools, verbatim}

\usepackage[usenames,dvipsnames]{xcolor}

\usepackage[shortlabels]{enumitem}
\SetEnumerateShortLabel{a}{\textup{(\alph*)}}
\SetEnumerateShortLabel{A}{\textup{(\Alph*)}}
\SetEnumerateShortLabel{1}{\textup{(\arabic*)}}
\SetEnumerateShortLabel{i}{\textup{(\roman*)}}
\SetEnumerateShortLabel{I}{\textup{(\Roman*)}}

\usepackage{soul}

\usepackage[initials,lite,alphabetic]{amsrefs} % amsrefs must be loaded after hyperref

\theoremstyle{plain}
\newtheorem{thm}{Theorem}[section]

\newtheorem{lem}[thm]{Lemma}

\newtheorem{prop}[thm]{Proposition}

\newtheorem*{proper*}{Property}

\newtheorem{cor}[thm]{Corollary}

\newtheorem{conj}[thm]{Conjecture}

\newtheorem*{lm*}{Lemma}
\newtheorem*{thm*}{Theorem}

\theoremstyle{definition}

\newtheorem*{df*}{Definition}

\newtheorem{exam}[thm]{Example}

\newtheorem{ex-notn}[thm]{Example/Notation}

\theoremstyle{remark}

\newtheorem{rem}[thm]{Remark}

\newtheorem{setting}[thm]{Setting}

\newtheorem*{acknowledgement*}{Acknowledgement}

\newtheorem*{ex*}{Example}
\newtheorem*{exer*}{Exercise}
\newtheorem*{rem*}{Remark}
\newtheorem*{prob*}{Problem}
\newtheorem*{prop*}{Proposition}

\def\ann{\operatorname{ann}}

\def\depth{\operatorname{depth}}

\def\hdepth{\operatorname{hdepth}} % Hilbert depth

\def\MG{\operatorname{MG}} % in Murai and Hibi's paper for the Hilbert function of lexsegment ideals
\def\projdim{\operatorname{proj\,dim}}

\def\sdepth{\operatorname{sdepth}} % Stanley depth

\def\frakm{\mathfrak{m}}

\def\frakp{\mathfrak{p}}

\def\KK{\mathbb{K}}

\def\NN{\mathbb{N}}

\def\ZZ{\mathbb{Z}}

\def\calD{\mathcal{D}}

\def\calH{\mathcal{H}}
\def\calI{\mathcal{I}}

\def\calP{\mathcal{P}}

\usepackage{bm}

\def\alert#1{\textcolor{Magenta}{#1}}

\def\ceil#1{\left\lceil #1 \right\rceil}

\def\deg{\operatorname{deg}}

\def\floor#1{\left\lfloor #1 \right\rfloor}

\def\Index#1{\emph{#1}}

\def\isom{\cong}

\def\sqr#1#2{{\vcenter{\hrule height.#2pt
\hbox{\vrule width.#2pt height#1pt \kern#1pt
\vrule width.#2pt}
\hrule height.#2pt}}}

\DeclareMathOperator{\m}{\mathbf{m}}

\let\oldmarginpar\marginpar
\renewcommand\marginpar[1]{\-\oldmarginpar[\raggedleft\footnotesize #1]%
{\raggedright\footnotesize #1}}

\def\opn#1#2{\def#1{\operatorname{#2}}} % to make operators
\opn\lex{lex}
\opn\rev{rev}
\opn\Lex{Lex}
\opn\GL{GL}
\opn\initial{in}

\begin{document}
\title{Lexsegment ideals of Hilbert depth 1}
\author{Yi-Huang Shen}
\address{Department of Mathematics, University of Science and Technology of China, Hefei, Anhui, 230026, China}
\address{Wu Wen-Tsun Key Laboratory of Mathematics, USTC, Chinese Academy of Sciences, China}
\thanks{This work is supported by the National Natural Science Foundation of China and the Fundamental Research Funds for the Central Universities (WK0010000017). 
Many important ideas of this work came from discussions with Professor J{\"u}rgen Herzog, to whom the author wants to express the most sincere gratitude. 
This work was partly completed while the author was visiting the Hong Kong University of Science and Technology. He thanks Professor Beifang Chen for useful discussions and hospitality during this visit.}
\email{yhshen@ustc.edu.cn}
\subjclass[2010]{ Primary
05E45, % Combinatorial aspects of simplicial complexes
05E40, % Combinatorial aspects of commutative algebra
06A07; % Combinatorics of partially ordered sets
Secondary
13C13, % Other special types of Commutative Algebra
05C70} % Factorization, matching, partitioning, covering and packing
\keywords{Stanley depth; Hilbert depth; Lexsegment ideals; Squarefree lexsegment ideals}
\date{\today}
\begin{abstract}
  Let $I\subset S=\KK[x_1,\dots,x_n]$ be a lexsegment ideal, generated by monomials of degree $d$. The main aim of this paper is to characterize when the Hilbert depth of $I$ will be $1$, in the standard graded case. In addition to this, we will give an estimate of depth of squarefree monomial ideals,
  generalizing a result of Popescu \cite{arXiv:1206.3977}. 
  We will also show that Stanley conjecture holds for squarefree stable ideals, in the multigraded case.
\end{abstract}
\maketitle

\section{Introduction}
Let $S=\KK[x_1,\dots,x_n]$ be a polynomial ring over a field $\KK$ of Krull dimension $n \ge 1$. It has a canonical $\ZZ^n$-grading. Let $M$ be a finitely generated $\ZZ^n$-graded $S$-module. Stanley \cite[5.1]{MR666158} conjectured that
\begin{equation}
\sdepth(M)\ge \depth(M).
\tag{$\bigstar$}
\label{Stanley_conj}
\end{equation}
Here, the Stanley depth $\sdepth(M)$ of $M$ will be introduced in Section 2 of this paper. Correspondingly, the depth $\depth(M)$ of $M$ is the common length of maximal $M$-sequences in the graded maximal ideal $(x_1,\dots,x_n)S$.

Herzog, Soleyman Jahan and Yassemi \cite[4.5]{MR2366164} showed that conjecture \eqref{Stanley_conj} implies the following combinatorial conjecture, due to Garsia {\cite[5.2]{MR597728}} and Stanley \cite[page 149]{MR526314} separately:
\[
\text{Every Cohen-Macaulay simplicial complex is partitionable.}
\]

Conjecture \eqref{Stanley_conj} remains open. From our point of view, the current related research includes at least the following three areas: 
\begin{itemize}
  \item Verify conjecture \eqref{Stanley_conj} in special cases. For example, the module $M$ is:
    \begin{itemize}
      \item an almost clean module \cite{arxiv.0712.2308};
      \item of the form $M=S/I$ where $I$ is an initial or final lexsegment ideal \cite[3.9]{arXiv:1010.5615}, or a generic or cogeneric Cohen-Macaulay ideal\cite[Theorems 3, 5]{MR1958009}, or a Cohen-Macaulay monomial ideal of codimension 2 \cite[2.4]{MR2366164},  or a Gorenstein monomial ideal of codimension 3 \cite[3.1]{MR2366164},  or the edge ideal of a complete $k$-partite graph \cite[2.8]{arXiv:1104.1018}. 
      \item a monomial ideal $I$ such that this $I$ is the intersection of four prime monomial ideals {\cite[4.2]{arXiv:1009.5646}},  or  the intersection of three monomial primary ideals {\cite[2.2]{arXiv:1107.3211}}, or an almost complete intersection ideal \cite[1.9]{arXiv:1112.4956}, or a general monomial ideal if the Krull dimension of the ring $n \le 5$ \cite[2.11]{MR2554495}.
    \end{itemize}
  \item Determine the Stanley depth of special modules, e.g, almost clean modules \cite{arxiv.0712.2308}, graded maximal ideals \cite{BHKTY}, monomial complete intersection ideals \cite{arXiv.org:0805.4461}, and some squarefree Veronese ideals \cite{arXiv:0907.1232}, \cite{project4} and \cite{arXiv:0911.5458}.
  \item Generalize the notion of Stanley depth, e.g., to cover depth \cite{MR2531663}, Hilbert depth \cite{MR2609292} and \cite{bruns-2009}, new depth \cite{arXiv:0908.3699}.
\end{itemize}
This list is definitely not complete, nor is it meant to be.

Many authors use the obstruction from Hilbert function to give upper bounds for Stanley depth. Hilbert depth, which studies the decomposition of the equivalent class of modules sharing the same (multigraded or standard graded) Hilbert series, consolidate this type of treatment. We will review the basics of Stanley depth and Hilbert depth in Section 2.

As an application of these notions, we give a sufficient condition for deciding the depth of $I/J$ where $J\subset I$ are two squarefree monomial ideals of $S$. This result, which generalizes a corresponding one in Popescu \cite{arXiv:1206.3977}, originates from a consideration of Stanley decomposition of $I/J$, as explained in Remark \ref{sd_exp}. Due to the still-openness of conjecture \eqref{Stanley_conj}, it is only resolved by applying Hilbert depth techniques, as shown in the proof of Theorem \ref{popescu-gen}.

In the standard graded setting, the Hilbert depth of the powers of the maximal ideals \cite{arXiv:1002.1400} and squarefree Veronese ideals \cite{arXiv:1008.4108} have been calculated and their relation has been studied \cite{arXiv:1106.3922}, from combinatorial point of view.  Except for these, the Hilbert depth of many interesting graded objects remains unknown.  

The aim of this paper is to characterize lexsegment ideals of Hilbert depth $1$. This will be completed in Section 5 as Theorem \ref{hdepth_lex}. Before that, we need some preparations which we outline here.

The lexsegment ideals that we will investigate in Theorem \ref{hdepth_lex} are special (strongly) stable ideals. We will go over these notions as well as their squarefree counterparts in Section 3. Several easy results related to their depth and Hilbert depth will be outlined as well.

Our proof for Theorem \ref{hdepth_lex} depends on the Gil Kalai correspondence between strongly stable ideals and squarefree strongly stable ideals. We will go over this relation in Section 4. This correspondence provides the underlying philosophy for connecting the Stanley depth and Hilbert depth of these two special types of monomial ideals, which we formulate as Conjecture \ref{conj2}.

After proving the main result in Section 5, we take a further study of the squarefree stable ideals in Section 6. We will calculate the Stanley depth of $S/I$ when $I$ is a squarefree strongly stable ideal in Remark \ref{sd_sqf_st_st} and show the Stanley conjecture \eqref{Stanley_conj} holds for squarefree stable ideals in Theorem \ref{sqf_st}.

\section{Stanley depth and Hilbert depth}
Let $S=\KK[x_1,\dots,x_n]$ be a polynomial ring over a field $\KK$ of Krull dimension $n$, with the lexicographic order $>_{\lex}$ induced by the ordering $x_1>x_2 > \cdots > x_n$. We consider two graded structures on $S$:
\begin{enumerate}[a]
  \item the multigrading, more precisely, the $\ZZ^n$-grading in which the degree of $x_i$ is the $i$th vector $e_i$ of the canonical basis;
  \item the standard grading over $\ZZ$ in which each $x_i$ has degree $1$.
\end{enumerate}

Following the convention in \cite{bruns-2009}, we will use the subscript $n$ to denote invariants associated with the multigrading, and the subscript $1$ for those associated with the standard grading.

Let $M$ be a finitely generated graded (in either standard graded or multigraded setting) $R$-module and $\frakm=(x_1,\dots,x_n)S$ be the graded maximal ideal of $S$. The grade of $\frakm$ on $M$ (also known as the depth of $M$) shall be simply written as $\depth(M)$.

A \Index{Stanley decomposition} of $M$ is a finite family
$\calD=(S_i,x_i)_{i\in \calI}$,
in which $x_i$ is a homogeneous element of $M$ and $S_i$ is a graded $\KK$-algebra retract of $S$ for each $i\in \calI$ such that $S_i\cap \ann(x_i)=0$, and 
\begin{equation}
  M=\bigoplus_i S_i x_i  
  \label{st-decomp}
\end{equation}
as a $\KK$-graded space. The direct sum on the right hand side of \eqref{st-decomp} carries a structure of an $R$-module, and thus has a well-defined $\depth$, called the \Index{Stanley depth} of this decomposition $\calD$. The \Index{Stanley depth} $\sdepth(M)$ of $M$ is the maximal depth of a Stanley decomposition of $M$.  We always set $\sdepth(0)=\infty$.

If $J\subsetneq I$ are two squarefree monomial ideals of $S$, consider the poset $P_{I\setminus J}$ of all squarefree monomials of $I\setminus J$ with the order given by divisibility. If $u\subseteq v$ are two monomials in $P_{I\setminus J}$, the interval $[u,v]$ is the set 
\[
\Set{w\in P_{I\setminus J}: \text{ $u$ divides $w$ and $w$ divides $v$}}.
\]
Herzog, Vladoiu and Zheng's method \cite[2.5]{arxiv.0712.2308} for squarefree monomial ideals can be easily checked to be equivalent to the following characterization:

\begin{lem}
  Let $k$ be a positive integer and $J\subsetneq I$ two squarefree monomial ideals of $S$. Then $\sdepth_n(I/J)\ge k$ if and only if $P_{I\setminus J}$ has a disjoint partition $\calP:P_{I\setminus J} = \coprod_{i=1}^l [u_i,v_i]$ such that the cardinalities $\left| v_i \right|\ge k$ for all $i$.
\end{lem}

With the additional help of \cite[3.3]{arXiv.org:0805.4461} and its proof, one can further assume (require) that in the partition $\calP$ above, for every $u_i$ such that $\deg(u_i)\le k$, one has $\deg(v_i)=k$.

\begin{cor}
  \label{rem-k}
  Let $k$ be a positive integer and $I$ a squarefree monomial ideal of $S$. If $J$ is another squarefree monomial ideal of $S$, generated by monomials of degrees $\ge k$, then $\sdepth_n(I)\ge k$ if and only if $\sdepth_n(I+J)\ge k$.
\end{cor}

The Hilbert decomposition that we shall review next, is a generalization of Stanley decomposition. To be more precise, a \Index{Hilbert decomposition} of $M$  is a finite family
$\calH=(S_i,s_i)_{i\in \calI}$,
such that $s_i \in \ZZ^m$ (where $m=1$ or $m=n$, respectively, depending on the grading), $S_i$ is a graded $\KK$-algebra retract of $R$ for each $i\in \calI$, and 
\begin{equation}
  M\isom \bigoplus_i S_i (-s_i)
  \label{hilb-decomp}
\end{equation}
as a graded $\KK$-vector space. The depth of the $R$-module on the right hand side of \eqref{hilb-decomp} is called the \Index{Hilbert depth} of $\calH$. The \Index{Hilbert depth} $\hdepth(M)$ of $M$ is the maximal depth of a Hilbert decomposition of $M$.

It follows easily from the definition that 
\[
\hdepth_1(I/J)\ge \sdepth_1(I/J)\ge \sdepth_n(I/J) = \hdepth_n(I/J)
\]
for any monomial ideals $J\subsetneq I$ in $S$; see \cite[2.8]{bruns-2009}.  Furthermore, we have the following facts for nonzero finitely generated graded module $M$.
\begin{enumerate}[a]
  \item Stanley conjecture holds in the standard graded case: 
    \[
    \sdepth_1(M)\ge \depth(M).
    \]
    This is a result of Baclawski and Garsia \cite{MR609203}, with a short proof in \cite[2.7]{bruns-2009}.
  \item When $\depth(M)\ge 1$, $\sdepth_n(M)\ge 1$ by \cite[2.13]{bruns-2009}. Since any nonzero monomial ideal $I$ satisfies $\depth(I)\ge 1$, one has $\sdepth_n(I)\ge 1$.
\end{enumerate}

It is worth noting that, in the standard graded case, Uliczka \cite[3.2]{MR2609292} proved the formula
\begin{equation}
  \hdepth_1(M)= \max\Set{u: (1-T)^u H_{M}(T) \text{ is positive}}.
  \label{Uliczka-formula}
\end{equation}
Here $H_{M}(T)$ is the Hilbert series of $M$, and a rational function is called \Index{positive} if its Laurent expansion at $0$ has only nonnegative coefficients.
In the following, we give an application of the formula \eqref{Uliczka-formula}.

\begin{setting}
  \label{setting1}
Let $I\supsetneq J$ be two squarefree monomial ideals,  generated by monomials of degrees $\ge d$ and $\ge d+1$ respectively.  Write $\rho_j(I\setminus J)$ for the number of all squarefree monomials of degree $j$ in $I\setminus J$.  
\end{setting}

It is known that $\depth(I/J)\ge d$ by \cite[1.1]{arXiv:1110.1963}.
Assume that $\depth_S(I/J)\ge t$, where $t$ is an integer such that $d\le t < n$. If 
\[
\rho_{t+1}(I\setminus J) < \alpha_t:= \sum_{i=0}^{t-d} (-1)^{t-d+i} \rho_{d+i}(I\setminus J),
\]
\cite[1.3]{arXiv:1206.3977} proved that $\depth_S(I/J)=t$ independently of the characteristic of $\KK$.
We generalize Popescu's result as follows:

\begin{thm}
  \label{popescu-gen}
  With the Setting \ref{setting1}, assume that $\depth_S(I/J)\ge t$, where $t$ is an integer such that $d\le t < n$.  If for some $k$ with $d+1\le k \le t+1$: 
  \begin{equation}
    \rho_k(I\setminus J) < \sum_{j=d}^{k-1} (-1)^{k-j+1}\binom{t+1-j}{k-j}\rho_j(I\setminus J),
    \label{ineq_k}
  \end{equation}
  then $\depth_S(I/J)=t$ independently of the characteristic of $\KK$.
\end{thm}

\begin{proof}
  In the standard graded setting, we have the following relation:
  \[
  \depth(I/J)\le \sdepth_1(I/J)\le \hdepth_1(I/J).
  \]
  Here, the first inequality is due to \cite[2.7]{bruns-2009} while the second inequality is simply by definition. Now it suffices to show that if inequality \eqref{ineq_k} holds for some suitable $k$, then $\hdepth_1(I/J)\le t$.

  We will take use of the formula \eqref{Uliczka-formula}, and calculate the Hilbert series $H_{I/J}(T)$ directly.  Notice that the canonical image of a monomial $u\in I$ in $I/J$ is nonzero if and only if $\sqrt{u}:=\prod_{x_i| u}x_i \in I\setminus J$. Thus
  \[
  H_{I/J}(T)=\sum_{j=d}^n \frac{\rho_j(I\setminus J)\cdot T^j}{(1-T)^j}.
  \]

  If $\hdepth_1(I/J)\ge t+1$, then $(1-T)^{t+1} H_{I/J}(T)$ is positive. Thus, its coefficient at each degree $k$ is nonnegative. When $d+1\le k \le t+1$, this is equivalent to saying
  \begin{equation}
    \sum_{j=d}^{k}(-1)^{k-j}\binom{t+1-j}{k-j} \rho_{j}(I\setminus J)  \ge 0.
    \label{pos_k}
  \end{equation}
  Relating \eqref{pos_k} with \eqref{ineq_k}, we complete the proof.
\end{proof}

Obviously, when inequality \eqref{ineq_k} holds for $k=t+1$, we will recover Popescu's result. 

\begin{rem}
  \label{sd_exp}
  Theorem \ref{popescu-gen} can be established by using Stanley depth, if the Stanley conjecture \eqref{Stanley_conj} holds in this case. Notice that when $\sdepth(I/J)\ge t+1$, $P_{I\setminus J}$ has a partition $P_{I\setminus J}=\coprod_i [u_i,v_i]$ such that for every $u_i$ with $\deg(u_i)\le t+1$, one has $\deg(v_i)=t+1$. Let $a_j:=|\Set{u_i:\deg(u_i)=j}|$ for $d\le j \le n$. Then $\rho_{k}(I\setminus J)=\sum_{j=d}^k a_j\binom{t+1-d}{j}$ for $d\le k \le t+1$.  
  It is not difficult to deduce from this fact those inequalities of \eqref{pos_k}. 
\end{rem}

When $\depth(I/J)\ge t+1$, we will have $\hdepth_1(I/J)\ge d+1$. Thus in the proof of Theorem \ref{popescu-gen}, after applying $k=d+1$ in \eqref{pos_k}, we get $\rho_{d+1}(I\setminus J) \ge (t+1-d)\rho_d(I)$. Now, as opposed to \cite[1.5]{arXiv:1206.3977}, we can get 

\begin{cor}
  With the Setting \ref{setting1}, suppose that $\depth(I/J) \ge d+2$, then  
  \[
  2\rho_d(I)\le \rho_{d+1}(I\setminus J) \le \rho_d(I) + \rho_{d+2}(I\setminus J).
  \]
  Whence, the condition $\rho_{d+2}(I\setminus J)=0$ forces $\rho_d(I)=\rho_{d+1}(I\setminus J)=0$.
\end{cor}

\section{Stable ideals and squarefree stale ideals}
For each monomial $u\in S$, let $\m(u)$ be the maximal integer $i$ such that $x_i$ divides $u$. If $I\subset S$ is a nonzero monomial ideal, define $\m(I)=\max\Set{\m(u): u\in G(I)}$, where $G(I)$ is the set of minimal monomial generators of $I$.
\begin{enumerate}[a]
  \item A monomial ideal $I$ is called \Index{stable} if for all monomials $u\in I$ and $i<\m(u)$ one has $x_i(u/x_{\m(u)})\in I$. 
    A squarefree monomial ideal $I$ is called \Index{squarefree stable} if for all squarefree monomials $u\in I$ and for all $i<\m(u)$ such that $x_i$ does not divide $u$ one has $x_i(u/x_{\m(u)})\in I$.
  \item $I$ is called \Index{strongly stable} if one has $x_i(u/x_j)\in I$ for all monomials $u\in I$ and all $i<j$ such that $x_j$ divides $u$.
    A squarefree monomial ideal $I$ is called \Index{squarefree strongly stable} if for all squarefree monomials $u\in I$ and for all $j<i$ such that $x_i$ divides $u$ and $x_j$ does not divide $u$ one has $x_j(u/x_i)\in I$.
  \item A (squarefree) monomial ideal $I$ is called \Index{(squarefree) lexsegment} if for all (squarefree) monomials $u\in I$ and all (squarefree) monomials $v$ with $\deg v= \deg u$ and $v>_{\lex} u$ one has $v\in I$. A set $V$ of monomials in $S_d$ is \Index{lexsegment} if the homogeneous ideal $S\cdot V$ is a lexsegment ideal.
\end{enumerate}

Obviously we have the implication
\[
\text{(squarefree) lexsegment} \implies \text{(squarefree) strongly stable} \implies \text{(squarefree) stable}.
\]

\begin{lem}
  \label{depth1}
  Let $I\subset S$ be a monomial ideal.
  \begin{enumerate}[a]
    \item If $I$ is stable, then $\depth(I)=n+1-\m(I)$.
    \item If $I$ is squarefree stable, then $\depth(I)=\min\Set{n-\m(u)+\deg(u):u\in G(I)}$.
  \end{enumerate}
\end{lem}

\begin{proof}
  When $I$ is stable,
  \[
  \projdim S/I=\max\Set{\m(u):u\in G(I)}
  \]
  by \cite[3.4(b)]{MR1890097}. We apply the graded version of Auslander-Buchsbaum formula
  \[
  \depth S/I +\projdim S/I=n
  \]
  and the well-known depth lemma \cite[1.2.9]{MR1251956 } to get the desired formula for $\depth(I)$. The squarefree case can be treated similarly, by using \cite[3.6(b)]{MR1890097}.
\end{proof}

Given positive integers $a$ and $d$, let
\[
a=\binom{\lambda_d}{d} + \cdots + \binom{\lambda_k}{k},
\]
where $k\ge 1$ and $\lambda_d > \cdots > \lambda_k\ge k\ge 1$, be the $d$th \Index{Macaulay representation} of $a$. One defines
\[
a^{\MG(d)}=\binom{\lambda_d+1}{d} + \cdots + \binom{\lambda_k+1}{k}
\]
and $0^{\MG(d)}=0$. Related, one also defines
\[
\partial_{d-1}(a)=\binom{\lambda_d}{d-1}+\cdots+\binom{\lambda_k}{k-1}
\]
and $\partial_{d-1}(0)=0$.

Suppose $I$ is a lexsegment ideal, generated by monomials of degree $d$. Let $u$ be the minimal monomial in $I_d$ with respect to $>_{\lex}$. Then $u$ can be written as
\begin{equation}
  \label{mini_lex}
u=x_1^{a_0-1}x_2^{a_1-a_0}\cdots x_k^{a_{k-1}-a_{k-2}}x_{k+1}^{a_k-a_{k-1}+1} x_n^{d-a_k},
\end{equation}
where $0<a_0\le a_1\le \cdots \le a_k \le d$ and $0\le k \le n-2$. The $(n-1)$th Macaulay representation of $\dim_\KK(I_d)=\mu(I)$ is
\[
\mu(I)=\binom{d-a_0+n-1}{n-1}+\cdots + \binom{d-a_k+n-1-k}{n-1-k},
\]
by \cite[C.10]{iarrobino1999gotzmann}.  Furthermore, the homogeneous piece $I_{d+1}=S_1\cdot I_d$ is again lexsegment. Equation (14) of \cite{MR2434473} says that
\begin{equation}
  \label{HLL} % Hilbert function of lexsegment ideals
  H(I,d+1)=H(I,d)^{\MG(n-1)}.
\end{equation}
It follows that the Hilbert function of $I$ is essentially determined by $\mu(I)=H(I,d)$, the minimal number of generators of $I$. 

\begin{cor}
  \label{lex_ideals}
  Suppose $I$ and $I'$ are two lexsegment ideals, generated in degree $d$ and $d'$ respectively, with $\mu(I)=\mu(I')$ and $d'\ge d$, then
  \begin{enumerate}[a]
    \item $I'=x_1^{d'-d}\cdot I$;
    \item the Hilbert series satisfies $H_{I'}(T)=T^{d'-d}H_{I}(T)$;
    \item the Hilbert function satisfies $H(I,d+\delta)=H(I',d'+\delta)$ for all $\delta\ge 0$;
    \item the Hilbert depth satisfies $\hdepth_1(I)=\hdepth_1(I')$ and $\hdepth_n(I)=\hdepth_n(I')$;
    \item the Stanley depth satisfies $\sdepth_1(I)=\sdepth_1(I')$ and $\sdepth_n(I)=\sdepth_n(I')$.
  \end{enumerate}
\end{cor}

\begin{proof}
  We only need to show part (a). Now, suppose $u$ in \eqref{mini_lex} is the minimal element in $I_d$ with respect to $>_{\lex}$.
  Since $\mu(I)=\mu(I')$, $x_1^{d'-d}u$ is the minimal element of $I_{d'}'$. For each monomial $v\in I_d$, one has $v>_{\lex} u$ in $I_d$ and thus $x_1^{d'-d}v>_{\lex} x_1^{d'-d}u$ in $I_{d'}'$. This is equivalent to saying that $x_1^{d'-d}I_d\subseteq I_{d'}'$. Since these two sets have the same cardinality, they must coincide.
\end{proof}

Let $V$ be a subspace of the $\KK$-vector space $S_d$. We write $\lex(V)\subset S_d$ for the $\KK$-vector space spanned by the lexsegment set $L\subset S_d$ of monomials with $|L|=\dim_\KK V$. The set $V$ is called a \Index{Gotzmann space} if $\dim_\KK(S_1 \cdot V)=\dim_\KK(S_1 \cdot \lex(V))$. The Gotzmann Persistence Theorem \cite{MR0480478} says that if $V$ is a Gotzmann space, then $S_1\cdot V$ is also a Gotzmann space. A homogeneous ideal $I$ of $S$ is \Index{Gotzmann} if $I_k$ is Gotzmann for all $k$. Lexsegment ideals are obviously Gotzmann.

\begin{cor}
  Suppose $I$ and $I'$ are Gotzmann ideals of $S$, and generated in degrees $d$ and $d'$ respectively, with $\mu(I)=\mu(I')$, then
  \begin{enumerate}[a]
    \item the Hilbert series satisfies $H_I(T)=T^{d-d'}H_{I'}(T)$;
    \item the Hilbert function satisfies $H(I,d+\delta)=H(I',d'+\delta)$ for all $\delta\ge 0$;
    \item the Hilbert depth satisfies $\hdepth_1(I)=\hdepth_1(I')$.
  \end{enumerate}
\end{cor}

Here is another direct application of \eqref{Uliczka-formula}:

\begin{lem}
  \label{hdepth_nec}
  Let $I$ be a monomial ideal of $S$, generated by monomials of degrees $\ge d$. If $\hdepth_1(I)\ge k\in \NN$, then $H(I,d+1)\ge k\cdot H(I,d)$.
\end{lem}

\begin{proof}
  The Hilbert series of $I$ is
  \[
  H_I(T)=H(I,d)T^d + H(I,d+1)T^{d+1}+\cdots.
  \]
  Direct computation shows that
  \[
  (1-T)^k H_I(T)= H(I,d)T^d + (H(I,d+1)-k \cdot H(I,d))T^{d+1} + \cdots.
  \]
  If $\hdepth_1(I)\ge k$, the coefficient $H(I,d+1)-k \cdot H(I,d)\ge 0$ by \eqref{Uliczka-formula}.
\end{proof}

\begin{conj}
  \label{conj1}
  Let $I$ be a stable ideal in $S$ of degree $d$, then $\hdepth_1(I)=\floor{H(I,d+1)/H(I,d)}$.
\end{conj}

This conjecture holds at least in the following two cases:
\begin{enumerate}[a]
  \item $I=(x_1,\dots,x_n)^d$ is the power of the graded maximal ideal, by \cite[1.2]{arXiv:1002.1400}.
  \item $I$ is lexsegment with $\hdepth_1(I)=1$, i.e., this lexsegment ideal satisfies $\mu(I)>\xi_{n-1}:=\sum_{j=1}^{n-1}\binom{2j-1}{j}$ by Theorem \ref{hdepth_lex}.
\end{enumerate}

\begin{exam}
  The Conjecture \ref{conj1} fails if $I$ is not stable. For example, one can take $I=(x^2,y^2)\subset \KK[x,y]$. Since $xy\not\in I$, $I$ is not stable. The Hilbert series of $I$ is
  \[
  H_I(T)=2T^2+ \sum_{j\ge 3}(j+1)T^j.
  \]
  Thus, $\hdepth_1(I)=1 < H(I,3)/H(I,2)=2$.
\end{exam}

\section{Shifting operations}

Let us review a pair of operators $(\sigma,\tau)$, due to Gil Kalai, that relates monomials with squarefree monomials in a larger polynomial ring. If $u=x_{i_1}x_{i_2}\cdots x_{i_d}$, where $i_1\le i_2 \le \cdots \le i_d$, we set $u^\sigma=x_{i_1}x_{i_2+1}\cdots x_{i_j+(j-1)}\cdots x_{i_d+(d-1)}$. The operator $\sigma : u\mapsto u^\sigma$ will be called the \Index{squarefree operator}. Its inverse $\tau$ is the map which associate each squarefree monomial $v=x_{i_1}\cdots x_{i_d}$, where $i_1<\cdots < i_d$ with the monomial $v^\tau=x_{i_1}x_{i_2-1}\cdots x_{i_j-(j-1)}\cdots x_{i_d-(d-1)}$. It is clear that the pair $(\sigma,\tau)$ establishes a bijection between $A_{n,d}$, the set of monomials in the variables $x_1,\dots,x_n$ of degree $d$, with $B_{n+d-1,d}$, the set of squarefree monomials in the variables $x_1,\dots,x_{n+d-1}$ of degree $d$.

If $I\subset S$ is a monomial ideal with $G(I)=\Set{u_1,\dots,u_s}$, we write $I^\sigma$ for the squarefree monomial ideal generated by the monomials $u_1^\sigma,\dots,u_s^\sigma$ in $\KK[x_1,\dots,x_m]$ where $m=\max\Set{\m(u)+\deg(u)-1:u\in G(I)}$.  Similarly, if $I\subset S$ is a squarefree monomial ideal with $G(I)=\Set{v_1,\dots,v_s}$, we write $I^\tau$ for the monomial ideal generated by $v_1^\tau,\dots,v_s^\tau$ in $S$. The operators $(\sigma,\tau)$ establishes a bijection between strongly stable ideals and squarefree strongly stable ideals:

\begin{lem}
  [{\cite[1.2, 1.4]{MR1803232}}]
  If $I\subset S$ is a strongly stable ideal with $G(I)=\Set{u_1,\dots,u_s}$, $I^\sigma$ is a squarefree strongly stable ideal with $G(I^\sigma)=\Set{u_1^\sigma,\dots,u_s^\sigma}$.  Conversely, if $I$ is a squarefree strongly stable ideal with $G(I)=\Set{v_1,\dots,v_s}$, $I^\tau$ is a strongly stable ideal with $G(I^\tau)=\Set{v_1^\tau,\dots,v_s^\tau}$.
\end{lem}

The bijection above restricts to a bijection between lexsegment ideals and squarefree lexsegment ideals:

\begin{lem}
  [{\cite[1.8]{MR1803232}}]
  \label{lex_sqf_lex}
  If $I \subset S$ is a lexsegment ideal, then $I^\sigma$ is a squarefree lexsegment ideal.
  Conversely, if $I$ is a squarefree lexsegment ideal in $S$, then $G(I^\tau) \subseteq S'=\KK[x_1,\dots,x_m]$ where $m=\max\Set{\m(u)-\deg(u)+1: u\in G(I)}$, and $I^\tau \cap S'$ is a lexsegment ideal in $S'$.
\end{lem}

If $I$ is a squarefree strongly stable monomial ideal, then $\beta_{i,i+j}(I)=\beta_{i,i+j}(I^\tau)$ for all $i$ and $j$ by \cite[2.2]{MR1803232}. Whence, all homological invariants that are expressible by the graded Betti numbers essentially coincide. For instance, we will have $\depth(I)=\depth(I^\tau)$ and $\hdepth_1(I)=\hdepth_1(I^\tau)$. Since powers of the graded maximal ideal and squarefree Veronese ideals are connected by this Gil Kalai correspondence, this gives a quick algebraic solution for the relation established in \cite{arXiv:1106.3922}; see also the discussion below.

This fact leads to the following conjecture:

\begin{conj}
  \label{conj2}
  Let $I\subset S=\KK[x_1,\dots,x_n]$ be a squarefree strongly stable ideal, generated by monomials of degree $d$. Then the following numbers coincide:
  \begin{enumerate}[i]
    \item $\sdepth_n(I)$,
    \item $\hdepth_n(I)$,
    \item $\hdepth_1(I)$,
    \item $\sdepth_n(I^\tau)$,
    \item $\hdepth_n(I^\tau)$,
    \item $\hdepth_1(I^\tau)$,
    \item $\floor{\frac{H(I^\tau,d+1)}{H(I^\tau,d)}}$.
  \end{enumerate}
\end{conj}

This conjecture was first suggested by Herzog, asking the equality of (i) and (iv). In this conjecture, the three equalities $\sdepth_n(I)=\hdepth_n(I)$, $\sdepth_n(I^\tau)=\hdepth_n(I^\tau)$ and $\hdepth_1(I)=\hdepth_1(I^\tau)$ are already known. In addition to Theorem \ref{hdepth_lex} that we will prove later, we have two additional supporting facts for this conjecture.
\begin{enumerate}[a]
  \item Let $I_{n,d}$ be the squarefree Veronese ideal, generated by all the degree $d$ squarefree monomials of $S=\KK[x_1,\dots,x_n]$. Now $1\le d \le n$ and $I_{n,d}$ is squarefree strongly stable with $(I_{n,d})^\tau=(x_1,\dots,x_{n-d+1})^d S$. Let $S'=\KK[x_1,\dots,x_{n-d+1}]$. Then we have the following relation that has been studied in detail in \cite{arXiv:1106.3922}:
    \[
    \hdepth_1(I_{n,d})=\hdepth_1( (I_{n,d})^\tau)=\hdepth_1( (x_1,\dots,x_{n-d+1})^d\cap S')+(d-1),
    \]
    whereas the formula
    \[
    \hdepth_1( (x_1,\dots,x_{n-d+1})^d \cap S')= \ceil{\frac{n-d+1}{d+1}}
    \]
    has been established in \cite[1.2]{arXiv:1002.1400}.

    On the other hand, for the Stanley depth, the conjecture
    \[
    \sdepth_n(I_{n,d})= \ceil{\frac{n-(d-1)}{d+1}}+(d-1)
    \]
    has been partially verified in \cite{arXiv:0907.1232}, \cite{project4} and \cite{arXiv:0911.5458}.
  \item Using direct computer verification, we know that a lexsegment ideal $I\subset \KK[x_1,\dots,x_4]$ of degree $d$ satisfies $\hdepth_1(I)\ge 3$ if and only if $\mu(I)\le 3$. On the other hand, when $d\le 9$, direct computer verification
    % project3-088.py
    shows that a squarefree lexsegment ideal $J\subset \KK[x_1,\dots,x_{d+3}]$ of degree $d$ satisfies $\sdepth_n(J)\ge d+2$ if and only if $\mu(J)\le 3$.
\end{enumerate}

The following example gives a counter-example for Conjecture \ref{conj2} when $I$ is not generated by monomials of the same degree:

\begin{exam}
  Take $I=\braket{x_1x_2,x_1x_3,x_1x_4,x_2x_3x_4,x_2x_3x_5}\subset S=\KK[x_1,\dots,x_5]$. Then $I$ is squarefree strongly stable. The corresponding ideal $I^\tau=\braket{x_1^2,x_1x_2,x_1x_3,x_2^3,x_2^2x_3}S$. Meanwhile, $\sdepth_n(I)=4> \sdepth_n(I^\tau)=3$ by the \texttt{SdepthLib.coc} library \cite{MR2531665} for \texttt{CoCoA} \cite{CocoaSystem}.
\end{exam}

The following example gives a counter-example for Conjecture \ref{conj2} when (squarefree) strongly stable condition is not satisfied:

\begin{exam}
  Let $J=\braket{x_1^2,x_2^2,x_3^2,x_1x_2,x_2x_3}\subset S=\KK[x_1,x_2,x_3]$. Since $x_1x_3\not\in J$, $J$ is not stable. The Hilbert series of $J$ is
  \[
  H_J(T)=\frac{5T^2-5T^3+T^5}{(1-T)^3}.
  \]
  It is not difficult to see that $\hdepth_1(J)=2$. On the other hand, $\sdepth_n(J)=1$ by the \texttt{SdepthLib.coc} library \cite{MR2531665} for \texttt{CoCoA} \cite{CocoaSystem}.
\end{exam}

For the time being, we don't have any counter example for Conjecture \ref{conj2} when $I$ is squarefree stable but not squarefree strongly stable. However, we are not very sure of its validity in this case.

\section{Lexsegment ideals}
If $k$ is a positive integer, write $\xi_k=\sum_{j=1}^k\binom{2j-1}{j}$.

\begin{lem}
  \label{sdepth_sqf_lex}
  Let $I$ be a squarefree lexsegment ideal, generated by monomials of degree $d$. Then $\sdepth_n(I)=d$ if and only if $\mu(I)> \xi_{n-d}$.
\end{lem}

\begin{proof}
  When $\mu(I)\le \xi_{n-d}$, $\sdepth_n(I)\ge d+1$ by \cite[1.1]{project7}. When $\mu(I)=\xi_{n-d}+1$, $\sdepth_n(I)=d$ by \cite[4.1]{project7}. Thus, when $\mu(I)>\xi_{n-d}$, $\sdepth_n(I)=d$ by \cite[2.3]{project7}.
\end{proof}

The following result is a generalization of \cite[3.6]{project7}.

\begin{prop}
  \label{inv1}
  The following conditions are equivalent for positive integers $x$ and $k$:
  \begin{enumerate}[a]
    \item $x\le \xi_k$;
    \item $\partial_{k-1}(x)\ge x$;
    \item $x^{\MG(k)}\ge 2x$.
  \end{enumerate}
\end{prop}

\begin{proof}
  Suppose $x=\binom{a_k}{k}+\cdots + \binom{a_i}{i}$ is the $k$th Macaulay representation of $x$.

  For the equivalence of $(b)$ and $(c)$, we just need to notice that $x^{\MG(k)}-x=\sum_{j=i}^k \binom{a_j+1}{j} -\sum_{j=i}^k \binom{a_j}{j} = \sum_{j=i}^k (\binom{a_j+1}{j}-\binom{a_j}{j})=\sum_{j=i}^k \binom{a_j}{j-1}=\partial_{k-1}(x)$.

  For the equivalence of $(a)$ and $(b)$, we recall that the direction $(a)\implies (b)$ is actually \cite[3.6]{project7}. The proof of $(b)\implies (a)$ shall be carried out similarly as follows.

  We induct on the positive integer $k$, with the case $k=1$ being trivial.

  Now, suppose $k>1$ and $x>\xi_k$. If $i>1$, for $1\le j < i$, we write $a_j=j-1$. Then $(a_k,a_{k-1},\dots,a_1)>_{\lex}(2k-1,2k-3,\dots,3,1)$ by \cite[4.2.7]{MR1251956}. In particular, $a_k\ge 2k-1$.

  If $a_k=2k-1$, we have $x'=\sum_{j=i}^{k-1}\binom{a_j}{j}> \xi_{k-1}$. Thus, by induction hypothesis, we have $\partial_{k-2}(x') < x'$ and equivalently $\partial_{k-1}(x) < x$.

  If $a_k>2k-1$, the inequality $\partial_{k-1}(x) < x$ is equivalent to
  \begin{equation}
    \binom{a_k}{k}-\binom{a_k}{k-1} > \sum_{j=i}^{k-1}\left[\binom{a_j}{j-1}-\binom{a_j}{j}\right].
    \label{11}
  \end{equation}
  In the left hand side of \eqref{11}, the difference $\binom{a_k}{k}-\binom{a_k}{k-1}$ is an increasing function for integer $a_k > 2k-1$ (see the proof of \cite[3.6]{project7}), thus has the minimum $\binom{2k}{k}-\binom{2k}{k-1}$. In the right hand side of \eqref{11}, we have $a_j\ge j$ and the difference $\binom{a_j}{j-1}-\binom{a_j}{j}$ is positive only for $a_j < 2j-1$. When $j\le a_k < 2j-1$, this difference is an increasing function for the integer $a_j$. Thus the right side of \eqref{11} has the maximum $\sum_{j=i}^{k-1}\left[ \binom{2j-2}{j-1}-\binom{2j-2}{j} \right]$. Now, it suffices to show that
  \begin{equation}
    \binom{2k}{k}-\binom{2k}{k-1} >
    \sum_{j=i}^{k-1}\left[ \binom{2j-2}{j-1}-\binom{2j-2}{j} \right].
    \label{22}
  \end{equation}
  Notice that previously we assume that $j\le a_j <2j-1$, thus $j>1$ and we  only need to consider the case when $i>1$ in the above inequality. Now \eqref{22} follows directly from the well-known recurrence relation of Catalan numbers. Recall that the $n$th Catalan number $C_n=\binom{2n}{n}-\binom{2n}{n+1}$, and it satisfies
  \[
  C_{n+1}=\sum_{j=0}^n C_i C_{n-j} \qquad \text{for $n\ge 0$},
  \]
  with $C_0=1$; see for instance \cite[14.7]{MR1871828}.
\end{proof}

\begin{thm}
  \label{hdepth_lex}
  Let $I\subset S$ be a nonzero lexsegment ideal, generated by monomials of degree $d$.
  \begin{enumerate}[a]
    \item In the standard graded case, the Hilbert depth $\hdepth_1(I)=1$ if and only if $\mu(I)>\xi_{n-1}$.
    \item In the multigraded case, the Hilbert depth $\hdepth_n(I)=1$ when $\mu(I) > \xi_{n-1}$.
  \end{enumerate}
\end{thm}

\begin{proof}
  \begin{enumerate}[a]
    \item In the standard graded case, we consider the following two sub-cases.
      \begin{enumerate}[i]
        \item If $\mu(I) > \xi_{n-1}$, Proposition \ref{inv1} says that $\mu(I)^{\MG(n-1)} < 2\cdot \mu(I)$, i.e., $H(I,d+1) < 2\cdot H(I,d)$. Thus, by Lemma \ref{hdepth_nec}, $\hdepth_1(I)\le 1$. Meanwhile, $\hdepth_1(I)\ge \depth(I)\ge 1$. Thus $\hdepth_1(I)=1$.
        \item If $\mu(I)\le \xi_{n-1}$, we may assume that $\mu(I)\ge n$, since otherwise $\m(I)=\mu(I)\le n-1$ and $\hdepth_1(I)\ge \depth(I)$ which is at least $2$ by Lemma \ref{depth1}. Now, with $\mu(I)\ge n$, the squarefree $I^\sigma \subseteq S'=\KK[x_1,\dots,x_{n+d-1}]$ is a squarefree lexsegment ideal by Lemma \ref{lex_sqf_lex}. Since $\mu(I^\sigma)=\mu(I)\le \xi_{n-1}$, $\sdepth_n(I^\sigma)\ge d+1$ by Lemma \ref{sdepth_sqf_lex}. Notice that $(I^\sigma)^\tau=IS''$. Thus $\hdepth_1(I) = \hdepth_1(I^\sigma) -(d-1) \ge \hdepth_n(I^\sigma)-(d-1)=\sdepth_n(I^\sigma)-(d-1)\ge 2$.
      \end{enumerate}
    \item In the multigraded graded case, when $\mu(I) > \xi_{n-1}$, one has $1\le \depth(I)\le \hdepth_n(I) \le \hdepth_1(I)\le 1$ by part (a). Thus, $\hdepth_n(I)=1$.
  \end{enumerate}
\end{proof}

\begin{rem}
  In Conjecture \ref{conj2}, we believe that the converse of Theorem \ref{hdepth_lex} (b) also holds, i.e., if $\mu(I)\le \xi_{n-1}$, then $\sdepth_n(I)\ge 2$. Except for computational evidence, we cannot establish this as a fact so far. However, we have the inequality
  \[
  \sdepth_n(I)\ge n-\floor{\frac{\mu(I)}{2}},
  \]
  which was established by Keller and Young \cite{KeYo2009} for squarefree monomial ideals, and by Okazaki \cite{Okazaki2009} for general monomial ideals. This implies that when $\mu(I)\le 2n-3$, $\sdepth_n(I)\ge 2$.
\end{rem}

\begin{conj}
  When $I$ is a lexsegment ideal of $S=\KK[x_1,\dots,x_n]$, generated by monomials of degree $d$, the Hilbert depth $\hdepth_1(I)$ is a decreasing function on $\mu(I)$.
\end{conj}

\section{Squarefree strongly stable ideals}

\begin{rem}
  \label{sd_sqf_st_st}
  Let $I\subseteq S=\KK[x_1,\dots,x_n]$ be a squarefree strongly stable monomial ideal (not necessarily generated by monomials of the same degree) and $\Delta$ the associated Stanley-Reisner simplicial complex. Then
  $\Delta$ is shifted, i.e., for each $F\in \Delta$, $i\in F$ and $j\in [n]$ with $j>i$, one has $(F\setminus \Set{i})\cup \Set{j}\in \Delta$. Thus
  $\Delta$ is nonpure shellable by \cite[11.3]{MR1401765}. Now the Stanley-Reisner ring $\KK[\Delta]=S/I$ is a clean ring by \cite{MR1239277}, and $\sdepth_n(R/I)=\depth(R/I)$ by \cite[1.3]{arxiv.0712.2308}. In particular, Stanley conjecture holds for $S/I$.
\end{rem}

If $I\subset S$ is a monomial ideal, Rauf \cite{arXiv:0812.2080} asked if
\begin{equation}
  \label{rauf-eqn}
  \sdepth_n(I)\ge 1+\sdepth_n(S/I).
\end{equation}
It holds for monomial complete intersections \cite[2.7]{arXiv:0812.2080} and intersections of two irreducible monomial ideals \cite[5.4]{MR2609185}.
In addition, if $1\le r \le e \le q$ are some integers such that $n=r+e+q$, and $\frakp_1=(x_1,\dots,x_r),\frakp_2=(x_{r+1},\dots,x_{r+e}),\frakp_3=(x_{r+e+1},\dots,x_{r+e+q})$ are disjoint prime ideals. Take $I=\frakp_1\cap \frakp_2\cap \frakp_3$. Then inequality \eqref{rauf-eqn} holds except possible in the case when either $r=e$ is even and $q$ is even, or $r$ is odd and $e=r+1$, see \cite[20]{MR2794292}.

Obviously, when $\sdepth_n(R/I)=\depth(R/I)$ (e.g, $I$ is squarefree strongly stable), inequality \eqref{rauf-eqn} is equivalent to the Stanley conjecture \eqref{Stanley_conj} for $I$:
  \[
  \sdepth_n(I)\ge \depth(I).
  \]

\begin{thm}
  \label{sqf_st}
  The Stanley conjecture holds for squarefree stable ideals. In particular, Rauf's inequality \eqref{rauf-eqn} holds for squarefree strongly stable ideals.
\end{thm}

\begin{proof}
  We prove by induction on the Krull dimension $n$. When $n=1$, all nonzero monomial ideals $I$ are principal with $\sdepth(I)=\depth(I)=1$.

  Now, let $n\ge 2$.  Let $I$ be a squarefree stable ideal of $S$, generated by monomials of degree $\ge d$ with $I_d\ne 0$. If $\m(I)< n$, take $S'=\KK[x_1,\dots,x_{m}]$ for $m=\m(I)$ and $I':=I\cap S'$. Then $I=I'S$ and $\m(I)=\m(I')$. Now $\depth(I)=\depth(I')+(n-m)$ and $\sdepth_n(I)=\sdepth_{m}(I')+(n-m)$ by \cite[3.6]{arxiv.0712.2308}. Thus, from the very beginning, we may assume that $\m(I)=n$.

  Obviously $G(I)=G_0\sqcup G_1$ where $G_0=\Set{u\in G(I): \m(u)< n}$ and $G_1=\Set{u\in G(I): \m(u)=n}$. $G_1\ne \emptyset$ by our assumption. If $G_1\cap S_{d} \ne \emptyset$, then $\depth(I)=d$ by Lemma \ref{depth1}. On the other hand, since $I$ is squarefree and generated by monomials of degree at least $d$, $\sdepth_n(I)\ge d$ by \cite[1.3, 3.1]{arxiv.0712.2308}. Thus, we are done in this case.

  Otherwise, $G_1 \subset S_{\ge \tilde{d}}$ where $\tilde{d}:=\depth(I)\ge d+1$. Write $S''=\KK[x_1,\dots,x_{n-1}]$ and $I''=I\cap S''$. Then $I''$ is squarefree stable with $G(I'')=G_0$ above. 
  By induction hypothesis and Lemma \ref{depth1}, we have $\sdepth_{n-1}(I'')\ge \depth_{S''}(I'')\ge \tilde{d}-1$. Thus, $\sdepth_n(I''S)\ge \tilde{d}$ by \cite[3.6]{arxiv.0712.2308}. Since $G_1\subset S_{\ge \tilde{d}}$, this implies that $\sdepth_n(I)\ge \tilde{d}=\depth(I)$ by Corollary \ref{rem-k}.
\end{proof}

%\bibliography{Bib}

% \bib, bibdiv, biblist are defined by the amsrefs package.

%\printindex
\end{document}